\documentclass[3p]{elsarticle}
\usepackage{amsmath,amsthm,amssymb,mathtools}
\usepackage{color}
\usepackage{mathtools}
\usepackage{hyperref}
\usepackage[capitalize,nameinlink]{cleveref}

\newcommand{\R}{\mathbb R}
\newcommand{\N}{\mathbb N}

\newcommand{\eps}{\varepsilon}

\renewcommand{\phi}{\varphi}

\newcommand{\mC}{\mathcal C}

\renewcommand{\:}{\mathrel{\coloneqq}}

\newcommand{\equivd}{\ensuremath{\vcentcolon\equiv}}
\newcommand{\OT}{{\Omega_{T}}}

\newcommand{\loc}{\textnormal{loc}}

\renewcommand{\d}{\mathrm{d}}
\newcommand{\dd}{\ensuremath{\,\mathrm{d}}}

\DeclareMathOperator{\sgnn}{sgn}
\newcommand{\sgn}[1]{\sgnn\left(#1\right)}

\DeclareMathOperator*{\essinf}{ess-\inf}

\crefname{ineq}{Inequality}{Inequalities}
\creflabelformat{ineq}{#2{(#1)}#3}
\crefname{assumption}{Assumption}{Assumptions}

\newtheorem{definition}{Definition}[section]
\newtheorem{theorem}{Theorem}[section]
\newtheorem{lemma}{Lemma}[section]
\newtheorem{corollary}{Corollary}[section]
\newtheorem{assumption}{Assumption}[section]
\newtheorem{remark}{Remark}[section]
\newtheorem*{maintheorem*}{Main Theorem}
\allowdisplaybreaks

\begin{document}

\title{A general result on the approximation of local conservation laws by nonlocal conservation laws: The singular limit problem for exponential kernels}


\author[G. M. Coclite]{Giuseppe Maria Coclite}
\address[G. M. Coclite]{Polytechnic University of Bari, Department of Mechanics, Mathematics, and Management, Via E. Orabona 4, 70125 Bari, Italy.}
\ead{giuseppemaria.coclite@poliba.it}

\author[J.-M. Coron]{Jean-Michel Coron}
\address[J.-M. Coron]{Université Pierre et Marie Curie, Laboratoire Jacques-Louis Lions, Place Jussieu 4, 75252 Paris, France.}
\ead{coron@ann.jussieu.fr}

\author[N. De Nitti]{Nicola De Nitti}
\address[N. De Nitti]{Friedrich-Alexander-Universität Erlangen-Nürnberg, Department of Mathematics, Chair of Applied Analysis (Alexander von Humboldt Professorship), Cauerstr. 11, 91058 Erlangen, Germany.}
\ead{nicola.de.nitti@fau.de}

\author[A. Keimer]{Alexander Keimer}
\address[A. Keimer]{Institute for Transportation Studies (ITS), University of California, Berkeley, 94720 Berkeley, CA, USA.}
\ead{keimer@berkeley.edu}

\author[1,2]{Lukas Pflug}
\address[1]{Friedrich-Alexander-Universität Erlangen-Nürnberg, Central Institute for Scientific Computing, Martensstr. 5a, 91058 Erlangen, Germany.}
\address[2]{Friedrich-Alexander-Universität Erlangen-Nürnberg,  Department of Mathematics, Chair of Applied Mathematics (Continuous Optimization), Cauerstr. 11, 91058 Erlangen, Germany.}
\ead{lukas.pflug@fau.de}

\begin{abstract}
We deal with the problem of approximating a scalar conservation law by a conservation law with nonlocal flux. As convolution kernel in the nonlocal flux, we consider an exponential-type approximation of the Dirac distribution. This enables us to obtain a total variation bound on the nonlocal term. By using this, we prove that the (unique) weak solution of the nonlocal problem converges strongly in \(C(L^{1}_{\text{loc}})\) to the entropy solution of the local conservation law. 
We conclude with several numerical illustrations which underline the main results and, in particular, the difference between the solution and the nonlocal term.

\end{abstract}
\maketitle

\section{Introduction}\label{sec:intro} 
Nonlocal conservation laws have been studied and analyzed quite intensively over the last decade from an application point of view with a particular focus on traffic flow \cite{bayen2020modeling,CGLM,scialanga,MR3890783,ridder2019traveling,huang2020stability}, supply chains \cite{sarkar,MR2737820,gong2019weak}, pedestrian flow/crowd dynamics \cite{colombo2020crowd}, opinion formation \cite{aletti2007first,piccoli2018sparse}, chemical engineering \cite{pflug2020emom,spinola2020model}, sedimentation \cite{MR2772627}, conveyor belts \cite{rossi2020well} and more. For the underlying dynamics existence and uniqueness \cite{keimer2,MR3670045,MR3890783,MR3818104,veeravalli,chiarello2019stability,crippa2013existence}, (optimal) control problems  \cite{groeschel,BCDKP2020,MR3986456,colombo,karafyllis2020analysis,coron}, and suitable numerical schemes \cite{aggarwal,chalons2018high,chiarello2020lagrangian,friedrich2018godunov,piccoli2013transport} have been analyzed.

In this work, ``nonlocal'' refers to the fact that the velocity \(V\) of the corresponding flux, i.e.,\ \(f(s)=V(s)s,\ s\in\R\),
does not depend on the solution locally at a given space point but on an integral of the solution on a neighborhood.

First in \cite{MR3342191} it has been observed that, at least numerically, there is some hope that the solution of the nonlocal conservation law converges to the local entropy solution when the nonlocal term approaches a Dirac delta. Positive results in this direction were obtained in \cite{MR1704419}, provided that the limit entropy solution is smooth and the convolution kernel is even, and in \cite{MR3944408} for a large class of nonlocal conservation laws under the assumption of having monotone initial data. Under the assumption that the initial datum has bounded total variation, is bounded away from zero and satisfies a one-sided Lipschitz condition, a positive result was obtained in \cite{1808.03529}. In \cite{MR4110434}, for an exponential weight in the nonlocal term, is was shown -- provided that the initial datum is bounded away from zero and has bounded total variation (but without monotonicity assumptions) -- that, the nonlocal solutions converge (up to subsequences) to weak solutions of the corresponding local conservation law; they also showed that the limit is the unique entropy solution  under the additional assumption that $V$ is an affine function. More recently, in \cite{bressanshen2020}, the result was extended to more general fluxes. 

A viscous nonlocal conservation law with kernel of exponential type was considered in \cite{CDKP2020}: as the nonlocal term together with the viscosity approaches zero, the sequence of solutions converges to the local entropy solution. The positive effect of viscosity in the nonlocal-to-local approximation process was previously studied in \cite{MR3961295,1902.06970} for more general compactly supported kernels (see also \cite{MR715133} in the case of more regular initial data and linear velocity). In contrast to the proof in \cite{MR3961295}, which was based on a priori estimates obtained by extensively using energy estimates for the heat kernel and the Duhamel representation formula, in \cite{CDKP2020}, the authors established an energy estimate on the nonlocal term by relying just on the special structure provided by the exponential kernel (as in Remark \ref{rem:localW} below) and used it to apply Tatar's compensated compactness theory (see \cite{T}).

In conclusion, although some progress has been made under quite restrictive assumptions, a general theory is missing. Even more, \cite{1808.03529} demonstrates via a counterexample that a total variation blow-up of the solution of the nonlocal conservation law can occur if the data is not bounded away from zero so that the standard methods via compactness in \(L^{1}\) seemed to be out of reach.

This is why, in this work, we focus on the nonlocal term: it turns out that the nonlocal term itself satisfies a local transport equation with nonlocal source (see \cref{lem:transport_nonlocal_rhs}) and we can use this to show a uniform total variation bound (see \cref{theo:total_variation_bound}). Thanks to the specific structure of the nonlocal term this directly implies that also the solution of the conservation law converges strongly in \(L^{1}\) (see \cref{theo:compactness} and \cref{cor:limit}). 

More precisely, in the present paper, we consider the following setting.
For a  nonlocal parameter \(\eta\in\R_{>0}\), let \(q_{\eta}\) be the unique weak solution (weak solutions are  unique in the nonlocal setup) of the nonlocal conservation law on \(\R\)
    \begin{align*}
\partial_t q_\eta(t,x)  + \partial_x\big(V\big(W_{\eta}[q_\eta](t,x))q_\eta(t,x)\big)   	&= 0	&& (t,x)\in\Omega_{T} \\ 
q_\eta(0,x) &= q_0(x)												 	&&  x\in\R
\intertext{with \(\OT \: (0,T) \times \R\), supplemented by the nonlocal term \(W_{\!\eta}\) with exponential weight}
    W_{\eta}[q](t,x)&\coloneqq\tfrac{1}{\eta}\int_{x}^{\infty} \exp(\tfrac{x-y}{\eta})q(t,y)\dd y &(t,x)\in\Omega_{T}
\end{align*}
and let \(q\) be the entropy solution of the corresponding local conservation law on \(\R\)
\begin{align}
    \partial_t q(t,x)  + \partial_x\big(V\big(q(t,x)\big)q(t,x)\big)   	&= 0	& (t,x)\in\Omega_{T} \label{eq:local_conservation_law}\\ 
q(0,x) &= q_0(x)								 	&  x\in\R.\label{eq:local_conservation_law_initial_datum}
\end{align}
For the ``local theory'' and corresponding Entropy solutions, we refer to \cite{MR1816648,godlewski,MR3468916, MR3443431}.
Finally, in \cref{theo:entropy}, we prove the convergence of the nonlocal solution to the local entropy solution when \(\eta\) approaches zero, i.e., the nonlocal term approaches a Dirac distribution
\[
q_{\eta} \stackrel{\eta \rightarrow 0}{\longrightarrow} q \ \text{ in } C\big([0,T];L^{1}_{\loc}(\R)\big).
\]
We do this by first analyzing the nonlocal term \(W_{\eta}[q_\eta]\) which -- as can be shown -- satisfies its own transport equation with a nonlocal source and possess a uniform total variation bound. Thanks to the relation \(\eta\partial_{2}W_{\eta}[q_{\eta}]\equiv W_{\eta}[q_{\eta}]-q_{\eta}\) it immediately follows also the strong convergence of \(q\) to a weak solution of the local conservation law. Then, we can use \cite{bressanshen2020} to obtain that the solution  is indeed also entropic. Even more, we also obtain that the nonlocal term \(W_{\eta}[q_{\eta}]\) also converges to the local entropy solution.

Our  ``nonlocal-to-local convergence'' result closes the gap between local and nonlocal modelling of phenomena governed by conservation laws; moreover, it provides a way of defining the entropy admissible solutions of local conservation laws as limits of weak solutions to nonlocal conservation laws, which usually do not require an entropy condition for uniqueness (see \cite{MR2679644,pflug,MR3818104,MR3890783}). This kind of singular limit would be an alternative to the classical vanishing viscosity approach (see \cite{MR3443431,MR1816648,MR3468916} and references therein). In the case of a nonlocal approximation without artificial viscosity, no smoothing phenomena happen and the character of the approximating equation remains somewhat ``hyperbolic'' (finite propagation of mass, but infinite propagation of information).

Such a convergence result would also give additional insights into questions related to control theory (see \cite{BCDKP2020}), in the spirit of \cite{MR2176274,MR2372489,MR2346380,MR2968071}. Showing control results for nonlocal conservation laws might be easier due to the fact that these equations are invertible in time, so that one can actually go back from a current state to the initial datum. Optimal control problems might also become mathematically more approachable as the problem with adjoint equations and shocks of the local equations prohibiting differentiability in a certain local framework might be resolvable in the nonlocal theory and one might then just consider the limit controls when the nonlocal term approaches a Dirac.


\section{Preliminary results on nonlocal conservation laws}
\label{sec:prelim}

\begin{definition}[The nonlocal conservation law and the weak solution]\label{defi:nonlocal_conservation_law_general}
Let \(T\in\R_{>0}\) be given. We consider for \(\eta\in\R_{>0}\) the following nonlocal conservation law in the ``density" \(q_{\eta}:\OT\rightarrow\R,\ \OT\:(0,T)\times\R\)
\begin{align}
\partial_t q_\eta(t,x)  + \partial_x\big(V\big(W_{\eta}[q_\eta](t,x))q_\eta(t,x)\big)   	&= 0	& (t,x)\in\Omega_{T} \label{eq:nonlocal_conservation_law} \\ 
q_\eta(0,x) &= q_0(x)												 	&  x\in\R\label{eqn:b}
\intertext{supplemented by the nonlocal term \(W_{\!\eta}\)}
    W_{\eta}[q_\eta](t,x)&\coloneqq\tfrac{1}{\eta}\int_{x}^{\infty} \exp(\tfrac{x-y}{\eta})q_\eta(t,y)\dd y &(t,x)\in\Omega_{T}.\label{defi:W_nonlocal}
\end{align}
We call \(q_{0}:\R\rightarrow\R\) \textbf{initial datum} and  \(W_{\eta}[q_\eta]:\OT\rightarrow\R\) the \textbf{nonlocal impact} affecting the \textbf{velocity function} \(V:\R\rightarrow\R\) of the nonlocal conservation law. 
We say that \(q_\eta\in C\big([0,T];L^{1}_{\loc}(\R)\big)\) is a weak solution for \(q_{0}\in L^{1}_{\loc}(\R)\) and $\eta \in \R_{>0}$ iff \(\forall \phi\in C^{1}_{\text{c}}((-42,T)\times\R)\) it holds that
\begin{equation}
\iint_{\OT}\partial_t\phi(t,x)q_\eta(t,x)+\partial_x\phi(t,x)V(W_{\!\eta}[q_\eta](t,x))q_\eta(t,x)\dd x\dd t+\int_{\R}\phi(0,x)q_{0}(x)\dd x =0 .
\label{eq:weak_solution_1}
 \end{equation}
\end{definition}
For the analysis and well-posedness, we require the following not restrictive assumptions:
\begin{assumption}[Assumptions on input data]\label{ass:input_datum}
The involved functions in \cref{defi:nonlocal_conservation_law_general} satisfy 
 \[ \bullet \ q_{0}\in L^{\infty}(\R;\R_{\geq0})\cap TV(\R)
   \qquad \bullet \ V\in W^{1,\infty}_{\loc}(\R):\ V'(s)\leq 0 \ \forall s \in \big(\essinf_{x\in \R} q_0(x),\|q_0\|_{L^\infty(\R)}\big).\]
\end{assumption}

\begin{theorem}[Existence and uniqueness of weak solutions and maximum principle]\label{theo:existence_uniqueness}
Given \cref{ass:input_datum} there exists a unique weak solution \(q\in C\big([0,T];L^{1}_{\text{loc}}(\R)\big)\cap L^{\infty}((0,T);L^{\infty}(\R))\cap L^{\infty}((0,T);TV(\R))\) of the nonlocal conservation law in \cref{defi:nonlocal_conservation_law_general} and the following maximum principle is satisfied
\begin{align}
    \essinf_{x\in\R} q_{0}(x)\leq q(t,x)\leq \|q_{0}\|_{L^{\infty}(\R)} \text{ a.e. } (t,x)\in\OT .\label{eq:uniform_bound}
\end{align}
\end{theorem}
\begin{proof}
See \cite[Theorem 2.20 \& Theorem 3.2 \& Corollary 4.3]{MR3670045}.
\end{proof}
\begin{remark}[Generalization of the assumptions on the velocity function \(V\)]\label{rem:velocity_generalization}
The assumption on \(V\) being monotonically decreasing (see \cref{ass:input_datum}) can be changed to \(V\) monotonically increasing as long as one also changes the nonlocal range for \(q\in C\big([0,T];L^{1}_{\loc}(\R)\big)\) as
\[
W_{\eta}[q](t,x)\: \int_{-\infty}^{x}\exp(\tfrac{y-x}{\eta})q(t,y)\dd y,\quad (t,x)\in\OT.
\]
Analogously, the results can be extended to hold also for non-positive initial datum when changing the nonlocal term accordingly. We do not go into details.

Even more when assuming that \(V'(s)s\) has a sign for all \(s\in\R\), one does not need even a maximum principle to be satisfied and thus the initial datum can be chosen arbitrarily in \(L^{\infty}(\R)\cap TV(\R)\) (no sign restrictions). However, then one does not obtain convergence of \(q_{\eta}\) but of \(W_{\eta}\) which remains essentially bounded and for which the total variation bound derived in \cref{theo:total_variation_bound} still holds. However, \cref{theo:entropy} is not directly applicable and we are left with that the limit is a weak solution. Compare also \cref{rem:TV_velocity}.
\end{remark}

\section{Total variation bound on the nonlocal term}
\label{sec:TV}

As we will tackle the convergence first in the nonlocal term, \(W_{\eta}[q_\eta]\), we deduce a transport equation with a nonlocal source which will enable us to study \(W_{\eta}[q_\eta]\) without \(q_\eta\) itself.
\begin{lemma}[The transport equation with nonlocal source satisfied by the nonlocal term]\label{lem:transport_nonlocal_rhs}
Given the dynamics in \cref{defi:nonlocal_conservation_law_general}, the nonlocal term \(W_{\eta}[q_{\eta}]\) as in \cref{eq:W} is Lipschitz-continuous and satisfies the following transport equation with nonlocal source in the strong sense
\begin{align}
    \partial_t W_\eta(t,x)+V(W_\eta(t,x))\partial_x W_\eta(t,x)&=-\tfrac{1}{\eta}\int_{x}^{\infty}\exp(\tfrac{x-y}{\eta})V'(W_\eta(t,y))\partial_y W_\eta(t,y)W_\eta(t,y)\dd y && (t,x)\in\OT\label{eq:W}\\
    W_{\eta}(0,x)&=\tfrac{1}{\eta}\int_{x}^{\infty}\exp(\tfrac{x-y}{\eta})q_{0}(y)\dd y && x\in\R.\label{eq:initial_W}
\end{align}
In particular, for \(\eta\in\R_{>0}\), we have \(W_{\eta}\in W^{1,\infty}(\OT)\).
\end{lemma}
\begin{proof}
We first show that \(W_{\eta}[q_{\eta}]\) is Lipschitz-continuous. To this end, recall the definition \cref{eq:W} and compute for \((t,x)\in\OT\)
\begin{align}
    \partial_{x}W_{\eta}[q_{\eta}](t,x)&=\partial_{x} \tfrac{1}{\eta}\int_{x}^{\infty}\exp(\tfrac{x-y}{\eta})q_{\eta}(t,y)\dd y=\tfrac{1}{\eta}W_{\eta}[q_{\eta}](t,x)-\tfrac{1}{\eta}q_{\eta}(t,x).\label{eq:WWxq}
\end{align}
However, as \(\eta\in\R_{>0}\), \(W_{\eta}[q_{\eta}]\in L^{\infty}((\OT))\) and  \(q_{\eta}\in L^{\infty}((\OT))\) thanks to \cref{theo:existence_uniqueness}, we obtain the uniform boundedness on the spatial derivative.
The time derivative is slightly more tricky. Due to the lack of regularity we use the method of characteristics analyzed in  \cite[Lemma 2.6]{MR3670045} to write down the solution \(q_{\eta}\) and have on \((t,x)\in\OT\)
\begin{align}
    \partial_{t}W_{\eta}[q_{\eta}](t,x)&=\partial_{t} \tfrac{1}{\eta}\int_{x}^{\infty}\exp(\tfrac{x-y}{\eta})q_{\eta}(t,y)\dd y\\
    &=\partial_{t}\tfrac{1}{\eta}\int_{x}^{\infty}\exp(\tfrac{x-y}{\eta})q_{0}(\xi(t,y;0))\partial_{2}\xi(t,y;0)\dd y\\
    &=\partial_{t}\tfrac{1}{\eta}\int_{\xi(t,x;0)}^{\infty}\exp\Big(\tfrac{x-\xi(0,z;t)}{\eta}\Big)q_{0}(z)\dd z\\
    &=-\tfrac{1}{\eta^{2}}\int_{\xi(t,x;0)}^{\infty}\exp\Big(\tfrac{x-\xi(0,z;t)}{\eta}\Big)q_{0}(z)\partial_{3}\xi(0,z;t)\dd z  -\tfrac{1}{\eta}q_{0}(\xi(t,x;0))\partial_{1}\xi(t,x;0).\label{eq:W_diff_computation}
\end{align}
Recalling some nice properties of the characteristics \cite[Lemma 2.6]{MR3670045} and in particular
\begin{align*}
    \partial_{3}\xi(0,\xi(t,y;0);t)&=V(W_{\eta}[q_{\eta}](t,y)) &&\forall (t,y)\in\OT\\
    \partial_{1}\xi(t,y;0)&=-\partial_{2}\xi(t,y;0)V(W_{\eta}[q_{\eta}](t,y))&& \forall (t,y)\in\OT
\end{align*}
we obtain by continuing \cref{eq:W_diff_computation}
\begin{align*}
\partial_{t}W[q_{\eta}](t,x)&=-\tfrac{1}{\eta^{2}}\int_{\xi(t,x;0)}^{\infty}\exp\Big(\tfrac{x-\xi(0,z;t)}{\eta}\Big)q_{0}(z)\partial_{3}\xi(0,z;t)\dd z  -\tfrac{1}{\eta}q_{0}(\xi(t,x;0))\partial_{1}\xi(t,x;0)\\
        &=-\tfrac{1}{\eta^{2}}\int_{x}^{\infty}\exp\Big(\tfrac{x-y}{\eta}\Big)q_{0}(\xi(t,y;0))\partial_{3}\xi(0,\xi(t,y;0)\partial_{2}\xi(t,y;0)\dd y\\
    &\qquad +\tfrac{1}{\eta}q_{0}(\xi(t,x;0))\partial_{2}\xi(t,x;0)V(W_{\eta}[q_{\eta}](t,x))\\
    &=-\tfrac{1}{\eta^{2}}\int_{x}^{\infty}\exp\Big(\tfrac{x-y}{\eta}\Big)q_{\eta}(t,y)V(W_{\eta}[q_{\eta}](t,y))\dd y + \tfrac{1}{\eta}q_{\eta}(t,x)V(W_{\eta}[q_{\eta}](t,x)).
\end{align*}
This expression is essentially bounded for \(\eta\in\R_{>0}\) so that we obtain the differentiability. Next, we show that the nonlocal operator indeed satisfies the Cauchy problem in \crefrange{eq:W}{eq:initial_W}.
Using the identity compute for \(W_{t}\) above we have for the left hand side of \cref{eq:W} and \((t,x)\in\OT\)
\begin{align*}
    &\partial_{t}W_{\eta}[q_{\eta}](t,x)+ V(W_{\eta}[q_{\eta}](t,x))\partial_{x}W_{\eta}[q_{\eta}](t,x)\\
    &=\tfrac{1}{\eta}q_{\eta}(t,x)V(W_{\eta}[q_{\eta}](t,x))
    -\tfrac{1}{\eta^{2}}\int_{x}^{\infty}\exp\big(\tfrac{x-y}{\eta}\big)q_{\eta}(t,y)V(W_{\eta}[q_{\eta}](t,y))\dd y\\
    &\qquad +V(W_{\eta}[q_{\eta}](t,x))\big(\tfrac{1}{\eta}W_{\eta}[q_{\eta}](t,x)-\tfrac{1}{\eta}q_{\eta}(t,x)\big)\\
    &=V(W_{\eta}[q_{\eta}](t,x))\tfrac{1}{\eta}W_{\eta}[q_{\eta}](t,x)\\& \qquad -\tfrac{1}{\eta^{2}}\int_{x}^{\infty}\exp\big(\tfrac{x-y}{\eta}\big)\big( W_{\eta}[q_{\eta}]t,y)-\eta\partial_{y} W_{\eta}[q_{\eta}](t,y))V(W_{\eta}[q_{\eta}](t,y)\big)\dd y\\
    &=V(W_{\eta}[q_{\eta}](t,x))\tfrac{1}{\eta}W_{\eta}[q_{\eta}](t,x)\\& \qquad -\tfrac{1}{\eta^{2}}\int_{x}^{\infty}\exp\big(\tfrac{x-y}{\eta}\big) W_{\eta}[q_{\eta}]t,y)V(W_{\eta}[q_{\eta}](t,y)\big)\dd y\\
    &\qquad+\tfrac{1}{\eta}\int_{x}^{\infty}\exp\big(\tfrac{x-y}{\eta}\big)\partial_{y} W_{\eta}[q_{\eta}](t,y)V(W_{\eta}[q_{\eta}](t,y)\big)\dd y\\
    &=-\tfrac{1}{\eta}\int_{x}^{\infty}\exp\big(\tfrac{x-y}{\eta}\big)V'(W_{\eta}[q_{\eta}](t,x))\partial_{y}W_{\eta}[q_{\eta}](t,y)W_{\eta}[q_{\eta}](t,y)\dd y
\end{align*}
where we have used two times the identity in \cref{eq:WWxq} and integration by parts. However, the last term is indeed the right hand side of \cref{eq:W}. The nonlocal term \(W_{\eta}\) also satisfies the initial datum in \cref{eq:initial_W} which is a direct consequence of the definition of \(W_{\eta}\) in \cref{defi:W_nonlocal} when plugging in \(t=0\) (this is possible as the solution is regular enough, i.e., \(q_{\eta}\in C\big([0,T];L^{1}_{\text{loc}}(\R)\big)\).
\end{proof}

\begin{remark}[Fully local equation in \(W_\eta\)]\label{rem:localW}
The transport equation in \(W_\eta\) in \cref{eq:W} with nonlocal source can also be transformed into a fully local equation (as in  \cite{CDKP2020}): 
\begin{align*}
 \partial_{t} W_\eta(t,x) + \partial_{x} \big(V(W_\eta(t,x))W_{\eta}(t,x)\big) &= \eta \partial_{tx}^{2} W_\eta(t,x) +\partial_{x} \big(V(W_\eta(t,x))\partial_{x} W_\eta(t,x)\big) && (t,x) \in \OT \\
W_\eta(0,x) &= \tfrac{1}{\eta}\int_x^\infty \exp(\tfrac{x-y}{\eta})q_{0}(y)\dd y && x \in  \R.
\end{align*}
\end{remark}

For our main theorem \cref{theo:total_variation_bound} where we prove a total variation bound on \(W_{\eta}\) uniform in \(\eta\) we require a density or stability result which enables us to smooth the solution. This result, stated below, is borrowed from \cite[Theorem 4.17]{MR3944408}. 

\begin{theorem}[Stability of the nonlocal conservation law w.r.t.\ the initial datum]\label{theo:stability}
Let \cref{ass:input_datum} hold, and  let $\mC_1,\mC_2 \in \R_{\geq0}$ be given such that \[Q(\mC_1,\mC_2) \: \big\{ u \in TV_{\text{loc}}(\R) \, : \, \|u\|_{L^\infty(\R)} \leq \mC_1 \wedge |u|_{TV(\R)} \leq \mC_2\big\}.\]
Let $q_0 \in Q(\mC_1,\mC_2)$ be given and denote by $q$ the solutions to the corresponding nonlocal conservation law.

Then, the solutions to the corresponding nonlocal conservation laws (denoted by \(q\)) satisfy the following \(C([0,T];L^{1}(\R))\) stability estimate, i.e.\ 
\begin{align*}
    \forall \epsilon\in\R_{>0}\ \exists \delta\in\R_{>0} : \forall\tilde{q}_{0} \in Q(\mC_1,\mC_2) \text{ with }\ \|q_{0}-\tilde{q}_{0}\|_{L^{1}(\R)}\leq \delta \ \implies \ \|q-\tilde{q}\|_{C([0,T];L^{1}(\R))}\leq \epsilon
\end{align*}
where $\tilde q$ is the solution to the corresponding nonlocal conservation law with initial datum $\tilde q_0$.
\end{theorem}
\begin{proof}
Almost the required result can be found in \cite[Theorem 4.17]{MR3944408} with the difference that the kernel of the nonlocal operator is supposed to be finite while here we have the exponential kernel (\cref{defi:W_nonlocal}). However, the changes for this result are minor, we do not go into details.
\end{proof}

The next theorem shows that the nonlocal term has a total variation which cannot increase over time.
\begin{theorem}[Total variation bound in the spatial component of \(W\) -- uniformly in \(\eta\)]\label{theo:total_variation_bound}
    The nonlocal term \(W_\eta\) defined in \cref{defi:W_nonlocal} but which also satisfies the identity demonstrated in \cref{lem:transport_nonlocal_rhs} satisfies -- uniformly in \(\eta\) a total variation bound, i.e.
    \[
    |W_{\eta}(t,\cdot)|_{TV(\R)}\leq |W_{\eta}(0,\cdot)|_{TV(\R)}\leq |q_{0}|_{TV(\R)}\ \forall  \eta\in\R_{>0}.
    \]
\end{theorem}
\begin{proof}
We take advantage of the stability result in \cref{theo:stability} which tells us that when smoothing \(q_{0}\) by \(q_{0}^{\eps}\equivd q_{0}\ast \phi_{\eps}\) with \(\phi_{\eps}\) being a standard mollifier \cite[C.4 Mollifiers]{leoni} with smoothing parameter \(\eps \in\R_{>0}\), the corresponding solution \(q_{\eta}^{\epsilon}\) will be close in the \(C([0,T];L^{1}(\R))\) topology. Additionally, as the initial datum is smooth, so is the corresponding solution (see \cite[Corollary 5.3]{MR3670045}) which we will denote by \(q_{\eta}^{\eps}\). We now prove the total variation bound. As the solution is smooth, the total variation coincides with the \(L^{1}\)-norm of the derivative and we obtain for \(t\in[0,T]\), which can be estimated as follows.
\begin{align}
    \tfrac{\d}{\d t}\int_{\R} |\partial_{x} W^\eps_\eta(t,x)|\dd x 
    &=\int_{\R} \sgn{\partial_{x} W^\eps_\eta(t,x)}\partial_{tx}^2 W^{\eps}_{\eta}(t,x)\dd x\notag\\
    &=-\int_{\R} \sgn{\partial_{x} W^\eps_\eta(t,x)}V(W^{\eps}_\eta(t,x))\partial^2_{xx} W^{\eps}(t,x)\dd x \notag\\ 
    &\qquad -\int_{\R} \sgn{\partial_{x} W^\eps_\eta(t,x)}V'(W^{\eps}_\eta(t,x))\big(\partial_{x} W^\eps_\eta(t,x)\big)^{2}\dd x\notag\\
    &\qquad +\tfrac{1}{\eta}\int_{\R} \sgn{\partial_{x} W^\eps_\eta(t,x)}V'(W^{\eps}_\eta(t,x))W^{\eps}_\eta(t,x)\partial_{x} W^\eps_\eta(t,x)\dd x\notag\\
    &\qquad -\tfrac{1}{\eta^{2}}\int_{\R} \sgn{\partial_{x} W^\eps_\eta(t,x)}\int_{x}^{\infty}\exp(\tfrac{x-y}{\eta})V'(W^{\eps}_\eta(t,y))\partial_{y} W^\eps_\eta(t,y)W^{\eps}_\eta(t,y)\dd y\dd x\notag\\
    &=\int_{\R} 2\delta_{0}(\partial_{x} W^\eps_\eta(t,x))V(W^{\eps}_\eta(t,x))\partial_{x} W^\eps_\eta(t,x)\partial^2_{xx} W^{\eps}_{\eta}(t,x)\dd x \notag \\ & \qquad +\int_{\R} \sgn{\partial_{x} W^\eps_\eta(t,x)}V'(W^{\eps})\big(\partial_{x} W^\eps_\eta(t,x)\big)^{2}\dd x\notag\\
    &\qquad -\int_{\R} \sgn{\partial_{x} W^\eps_\eta(t,x)}V'(W^{\eps}_\eta(t,x))\big(\partial_{x} W^\eps_\eta(t,x)\big)^{2}\dd x\notag\\ 
    & \qquad  +\tfrac{1}{\eta}\int_{\R} \sgn{\partial_{x} W^\eps_\eta(t,x)}V'(W^{\eps}_\eta(t,x))W^{\eps}_\eta(t,x)\partial_{x} W^\eps_\eta(t,x)\dd x\notag\\
    &\qquad -\tfrac{1}{\eta^{2}}\int_{\R} \sgn{\partial_{x} W^\eps_\eta(t,x)}\int_{x}^{\infty}\exp(\tfrac{x-y}{\eta})V'(W^{\eps}_\eta(t,y))\partial_{y} W^\eps_\eta(t,y)W^{\eps}_\eta(t,y)\dd y\dd x\notag\\
    &\qquad+\tfrac{1}{\eta}\int_{\R} |\partial_{x} W^\eps_\eta(t,x)|V'(W^{\eps}_\eta(t,x))W^{\eps}_\eta(t,x)\dd x \notag
    \\
    &\qquad -\tfrac{1}{\eta^{2}}\int_{\R} V'(W^{\eps}_\eta(t,y)\partial_{y} W^\eps_\eta(t,y)W^{\eps}_\eta(t,y) \int_{-\infty}^{y}\sgn{\partial_{x} W^\eps_\eta(t,x)}\exp(\tfrac{x-y}{\eta})\dd x\dd y\notag\\
    &\leq\tfrac{1}{\eta}\int_{\R} |\partial_{x} W^\eps_\eta(t,x)|V'(W^{\eps}_\eta(t,x))W^{\eps}_\eta(t,x)\dd x \notag\\ 
    & \qquad -\tfrac{1}{\eta^{2}}\int_{\R}V'(W^{\eps}_\eta(t,y))|\partial_{y} W^\eps_\eta(t,y)|W^{\eps}_\eta(t,y) \int_{-\infty}^{y}\exp(\tfrac{x-y}{\eta})\dd x\dd y\label{eq:42}\\
    &=\tfrac{1}{\eta}\int_{\R} |\partial_{x} W^\eps_\eta(t,x)|V'(W^{\eps}_\eta(t,x))W^{\eps}_\eta(t,x)\dd x \notag\\ 
    &\qquad -\tfrac{1}{\eta}\int_{\R}V'(W^{\eps}_\eta(t,y))|\partial_{y} W^\eps_\eta(t,y)|W^{\eps}_\eta(t,y) \exp(\tfrac{y-y}{\eta})\dd y\notag\\
    & = 0.\notag
\end{align}

We thus obtain 
\begin{align*}
    |W_{\eta}^{\epsilon}(t,\cdot)|_{TV(\R)}\leq |W_{\eta}^{\epsilon}(0,\cdot)|_{TV(\R)} \leq |q_{0}|_{TV(\R)},
\end{align*}
where the last inequality follows from the assumption on \(q_{0}\in TV(\R)\) as stated in \cref{ass:input_datum} and the definition of the initial value for \(W_\eta\) as in \cref{eq:initial_W}:
\begin{align*}
    \|W_{\eta}^{\eps}(0,\cdot)\|_{TV(\R)}&=\!\!\!\sup_{\substack{\psi\in C^{1}_{\text{c}}(\R):\\ \|\psi\|_{L^{\infty}(\R)}\leq 1}}\!\!\!\int_{\R}\psi'(x)W_{\eta}^\eps[q_{0}^{\eps}](x)\dd x =\!\!\!\sup_{\substack{\psi\in C^{1}_{\text{c}}(\R):\\
    \|\psi\|_{L^{\infty}(\R)}\leq 1}}\!\!\! \int_{\R}\psi'(x)\tfrac{1}{\eta}\int_{\R_{>0}}\exp(\tfrac{x-y}{\eta})q_{0}^\eps(y)\dd y\dd x \\
    &=\!\!\!\sup_{\substack{\psi\in C^{1}_{\text{c}}(\R):\\ \|\psi\|_{L^{\infty}(\R)}\leq 1}}\!\!\!\int_{\R}\psi'(x)\tfrac{1}{\eta}\int_{\R_{<0}}\exp(\tfrac{z}{\eta})q_{0}^\eps(x-z)\dd y\dd x 
    =\!\!\!\sup_{\substack{\psi\in C^{1}_{\text{c}}(\R):\\ \|\psi\|_{L^{\infty}(\R)}\leq 1}}\sup_{z\in \R_{<0}}\int_{\R}\psi'(x+z)q_{0}^\eps(x)\dd y\\ 
    &=\!\!\!\sup_{\substack{\psi\in C^{1}_{\text{c}}(\R):\\ \|\psi\|_{L^{\infty}(\R)}\leq 1}}\!\!\!\int_{\R}\psi'(x)\int_{\R} \phi_\eps(x-y)q_{0}(x) \dd x\dd y
    \leq\sup_{y\in\R}\int_{\R}\phi_\eps(x-y)\!\! \!\!\!\sup_{\substack{\psi\in C^{1}_{\text{c}}(\R):\\ \|\psi\|_{L^{\infty}(\R)}\leq 1}}\int_{\R} \psi'(z)q_{0}(z)\dd z \dd x\\       
    &\leq |q_{0}|_{TV(\R)}.
\end{align*}
\end{proof}
\begin{remark}[Total variation bound and the required assumptions on the velocity \(V\)]
\label{rem:TV_velocity}
The key step in the proof of the  total variation bound stated in  \cref{theo:total_variation_bound} can be located in the estimate around \cref{eq:42}. Reconnecting to \cref{rem:velocity_generalization} it is enough to assume the velocity to satisfy \(V'(s)s\leq 0\ \forall s\in\R\) to obtain the uniform total variation bound without any sign restriction on the initial datum.
\end{remark}

\section{Compactness argument and proof of the convergence result}
\label{sec:convergence}

\begin{theorem}[Compactness of \(W_{\eta}\) in {\(C\big([0,T];L^{1}_{\text{loc}}(\R)\big)\)}]\label{theo:compactness}
The set \(\big(W_{\eta}\big)_{\eta\in\R_{>0}}\subseteq C\big([0,T];L^{1}_{\text{loc}}(\R)\big)\) of solutions to \crefrange{eq:W}{eq:initial_W} is compactly embedded into \(C\big([0,T];L^{1}_{\loc}(\R)\big),\) i.e.\ 
\begin{align*}
    \Big\{W_{\eta}\in C\big([0,T];L^{1}_{\loc}(\R)\big):\ W_{\eta} \text{ satisfies \crefrange{eq:W}{eq:initial_W}},\ \eta\in\R_{>0}\Big\}\overset{\text{c}}{\hookrightarrow} C\big([0,T];L^{1}_{\loc}(\R)\big).
\end{align*}
\end{theorem}
\begin{proof}
The proof consists basically of applying the Ascoli theorem in \cite[Lemma 1]{simon}. We state the details in the following.
Let \(B\) be a Banach space. 

Then, a set \(F\subset C([0,T];B)\) is relatively compact in \(C([0,T];B)\) iff
\begin{itemize}
    \item \(F(t)\:\big\{f(t)\in B: f\in F\}\) is relatively compact in \(B\ \forall t\in[0,T]\).
    \item \(F\) is uniformly equi-continuous, i.e.\ 
    \[
    \forall \sigma\in\R_{>0}\ \exists \delta\in\R_{>0}\ \forall f\in F\ \forall (t_{1},t_{2})\in [0,T]^{2} \text{ with } |t_{1}-t_{2}|\leq \delta:\ \|f(t_{1})-f(t_{2})\|_{B}\leq \sigma.
    \]
\end{itemize}
We start with setting \(B=L^{1}_{\loc}(\R)\) and \(F(t)\:\{W_{\eta}(t,\cdot)\in L^{1}_{\loc}(\R): \eta\in\R_{>0}\}\).
Thanks to \cref{theo:total_variation_bound} we know that \(W_{\eta}(t,\cdot)\) has a uniform total variation bound and by \cite[Theorem 13.35]{leoni}, the set \(F(t)\) is compact in \(L^{1}_{\loc}(\R)\), i.e.\ \[F(t)\overset{\text{c}}\subseteq L^{1}_{\loc}(\R),\quad \forall t\in[0,T].\]
It remains to show the second point, the uniformly equi-continuity. To this end, we again smooth the initial datum \(q_{0}\) by a \(q_{0}^{\eps}\) for \(\eps\in\R_{>0}\) as in the proof of \cref{theo:total_variation_bound} and call the corresponding smooth nonlocal term \(W_{\eta}^{\eps}\) for an \(\eta\in\R_{>0}\). Then, we can estimate 
\begin{align*}
    \big\|W_{\eta}^{\eps}(t_{1},\cdot)-W_{\eta}^{\eps}(t_{2},\cdot)\big\|_{L^{1}(\R)}
    &=\left\|\int_{t_{2}}^{t_{1}}\partial_{t}W_{\eta}^{\eps}(s,\cdot)\dd s\right\|_{L^{1}(\R)}
    \intertext{plugging \cref{eq:W} in}
    &\leq \left\|\int_{t_{2}}^{t_{1}}V(W_{\eta}^{\eps}(s,\cdot))\partial_{2}W^\eps_{\eta}(s,\cdot)\dd s\right\|_{L^{1}(\R)}\\
    &\quad +\left\|\int_{t_{2}}^{t_{1}}\tfrac{1}{\eta}\int_{\ast}^{\infty}\exp(\tfrac{\ast-y}{\eta})V'(W_{\eta}^{\eps}(s,y))\partial_yW_{\eta}^{\eps}(s,y)W^{\eps}_{\eta}(s,y)\dd y\dd s\right\|_{L^{1}(\R)}
    \intertext{integrating by parts}
    &\leq \|V\|_{L^{\infty}((0,\|q_{0}\|_{L^{\infty}(\R)}))}|W_{\eta}^{\epsilon}|_{L^{\infty}((0,T);TV(\R))}|t_{1}-t_{2}|\\
    &\quad +\|V'\|_{L^{\infty}((0,\|q_{0}\|_{L^{\infty}(\R)}))}\|W_{\eta}^{\epsilon}\|_{L^{\infty}((0,T);L^{\infty}(\R))}|W_{\eta}^{\epsilon}|_{L^{\infty}((0,T);TV(\R))}|t_{1}-t_{2}|\\
    \intertext{applying \cref{theo:total_variation_bound} and \cref{eq:uniform_bound}}
    &\leq \big(\|V\|_{L^{\infty}((0,\|q_{0}\|_{L^{\infty}(\R)}))}+\|V'\|_{L^{\infty}((0,\|q_{0}\|_{L^{\infty}(\R)}))}\|q_{0}\|_{L^{\infty}(\R)}\big)|q_{0}|_{TV(\R)}|t_{1}-t_{2}|.
    \end{align*}
As this is a uniform bound in \(\eta\in\R_{>0}\) and \(\eps\in\R_{>0}\) we have the uniform equi-continuity so that we obtain by applying Ascoli's theorem indeed the claimed compactness.
\end{proof}

As a direct result, from the strong convergence of \(W_{\eta}\) we have also the strong convergence of \(q_{\eta}\) to a weak solution of the local conservation law as the following corollary states:
\begin{corollary}[Limit of \(q_{\eta}\) and \(W_{\eta}\) are weak solution to the local equation]\label{cor:limit}
For every sequence \((\eta_{k})_{k\in\N_{\geq1}}\subset\R_{>0}\) with \(\lim_{k\rightarrow\infty}\eta_k=0\) there exists a subsequence (for reasons of convenience again denoted by \(\eta_{k}\)) and a function \(q^{*}\in C\big([0,T];L^{1}_{\loc}(\R)\big)\) so that 
the solution \(q_{\eta_{k}}\in C\big([0,T];L^{1}_{\loc}(\R)\big)\) of the nonlocal conservation law as given in \cref{defi:nonlocal_conservation_law_general} converges in \(C\big([0,T];L^{1}_{\text{loc}}(\R)\big)\) to the limit point \(q^{*}\) and so does the nonlocal term \(W_{\eta_{k}}\) as given in \cref{defi:W_nonlocal}. Additionally, \(q^{*}\) is a weak solution of the local conservation law  \crefrange{eq:local_conservation_law}{eq:local_conservation_law_initial_datum}. In equations,
\begin{align*}
    \lim_{\eta\rightarrow 0 }\|q_{\eta}-q^{*}\|_{C([0,T];L^{1}_{\text{loc}}(\R))}=0 \ \wedge\ \lim_{\eta\rightarrow 0 }\|W_{\eta}-q^{*}\|_{C([0,T];L^{1}_{\text{loc}}(\R))}=0,
\end{align*}
when \(q^{*}\) satisfies
\(\forall \phi\in C^{1}_{\text{c}}((-42,T)\times\R)\)
\begin{equation}
\iint_{\OT}\partial_t\phi(t,x)q^{*}(t,x)+\partial_x\phi(t,x)V(q^{*}(t,x))q^{*}(t,x)\dd x\dd t+\int_{\R}\phi(0,x)q_{0}(x)\dd x =0 .
\label{eq:weak_solution_local}
 \end{equation}
\end{corollary}

\begin{proof}
Thanks to \cref{theo:compactness}, \(\mathcal{W}\:\{W_{\eta_{k}}; k\in\N_{\geq1}\}\subset C\big([0,T];L^{1}_{\loc}(\R)\big),\) i.e., the set \(\mathcal{W}\) is compact in \(C\big([0,T];L^{1}_{\loc}(\R)\big)\) and there exists a limit point \(q^{*}\in C\big([0,T];L^{1}_{\loc}(\R)\big)\) so that we obtain
\[
\lim_{k\rightarrow\infty}\|W_{\eta_{k}}-q^{*}\|_{C([0,T];L^{1}_{\loc}(\R))}=0.
\]
The identity in \cref{eq:WWxq}  directly implies
\begin{align*}
    \|W_{\eta_{k}}(t,\cdot)-q_{\eta_{k}}(t,\cdot)\|_{L^1(\R)} &= \eta_{k} |W_{\eta_{k}}(t,\cdot)|_{TV(\R)} \leq \eta_{k}|q_0|_{TV(\R)}
\intertext{and thus we also obtain}
\lim_{k\rightarrow\infty}\|q_{\eta_{k}}-q^{*}\|_{C([0,T];L^1_{\loc}(\R))}&=0.
\end{align*}
It remains to be shown that \(q^{*}\) is indeed a weak solution. This directly follows from the strong convergence of \(q_{\eta_{k}}\) to \(q^{*}\) in \(C\big([0,T];L^{1}_{\text{loc}}(\R)\big)\) and due to the essential and uniform bound on \(q_{\eta}\) as given in \cref{theo:existence_uniqueness} in \cref{eq:uniform_bound}. 
\end{proof}
However, the previous result can actually be strengthened and indeed we obtain that the limit \(q^{*}\) is unique (in particular, every subsequence converges)  and that this limit is the weak entropy solution of the corresponding local conservation law.

\begin{theorem}[Convergence to the Entropy solution]\label{theo:entropy}
Given \cref{ass:input_datum} the nonlocal term \(W_{\eta}[q_{\eta}]\) and the corresponding nonlocal solution \(q_{\eta}\in C\big([0,T];L^{1}_{\loc}(\R)\big)\) of the nonlocal conservation law \cref{defi:nonlocal_conservation_law_general} converges in \(C\big([0,T];L^{1}_{\loc}(\R)\big)\) to the Entropy solution of the corresponding local conservation law (see \crefrange{eq:local_conservation_law}{eq:local_conservation_law_initial_datum}).
\end{theorem}

\begin{proof}
This is a direct consequence of  the convergence of \(W_{\eta},q_{\eta}\) to a weak solution of the local conservation laws in \(C\big([0,T];L^{1}_{\loc}(\R)\big)\), \cref{cor:limit} and of \cite{bressanshen2020}.
Therein, by taking advantage of the minimal entropy condition in  \cite{otto,panov1994uniqueness}, it is shown that a solution \(q_{\eta}\) of the nonlocal conservation law in \cref{defi:nonlocal_conservation_law_general} with uniform \(TV\) bound converges to the entropy solution of the local problem. However, when checking the proof carefully, it turns out that it suffices to assume that the solution \(q_{\eta}\) converges strongly to a weak solution \(q^{*}\).
\end{proof}

\section{Numerical illustrations}\label{sec:numerics}

\begin{figure}
    \centering
    \includegraphics[scale=0.8,clip,trim=0 25 15 0]{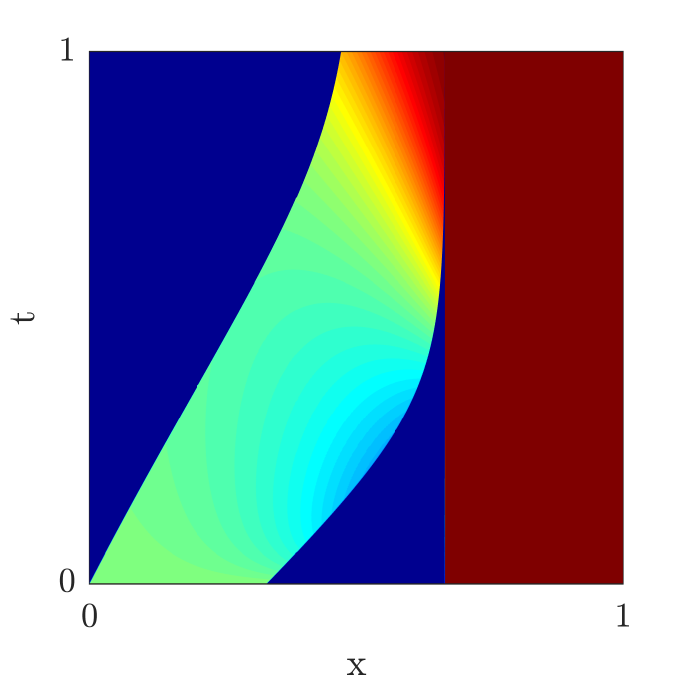}
    \includegraphics[scale=0.8,clip,trim=22 25 15 0]{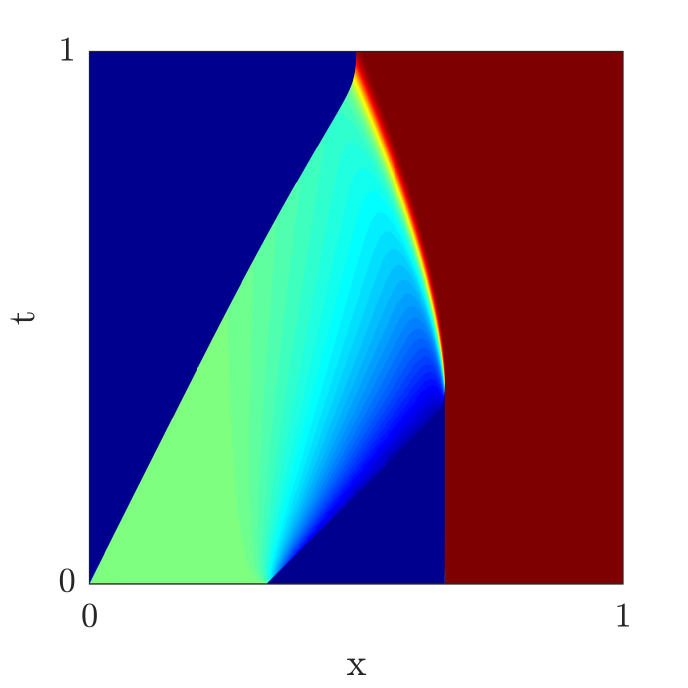}
    \includegraphics[scale=0.8,clip,trim=22 25 15 0]{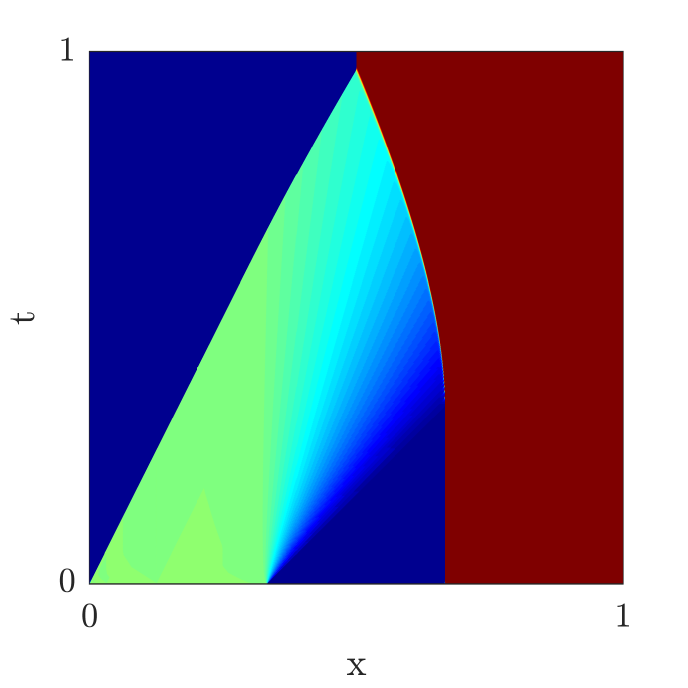}

        \includegraphics[scale=0.8,clip,trim=0 0 15 0]{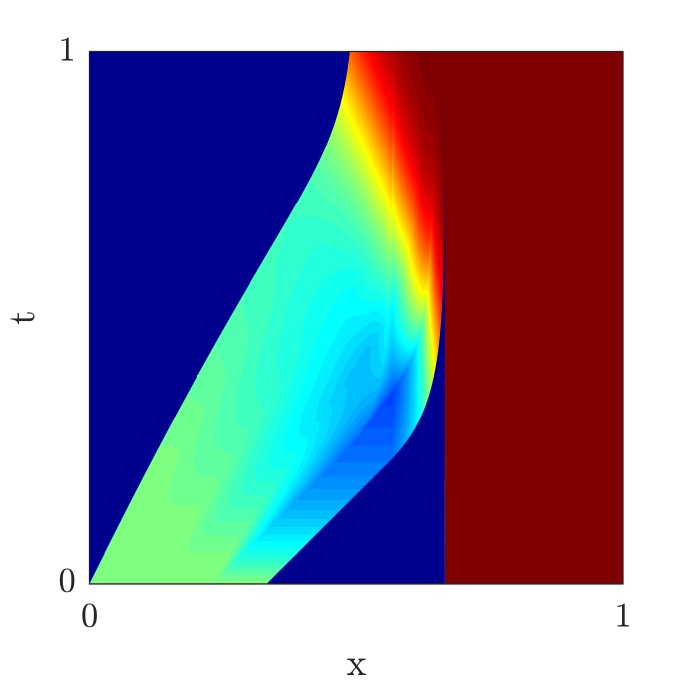}
    \includegraphics[scale=0.8,clip,trim=22 0 15 0]{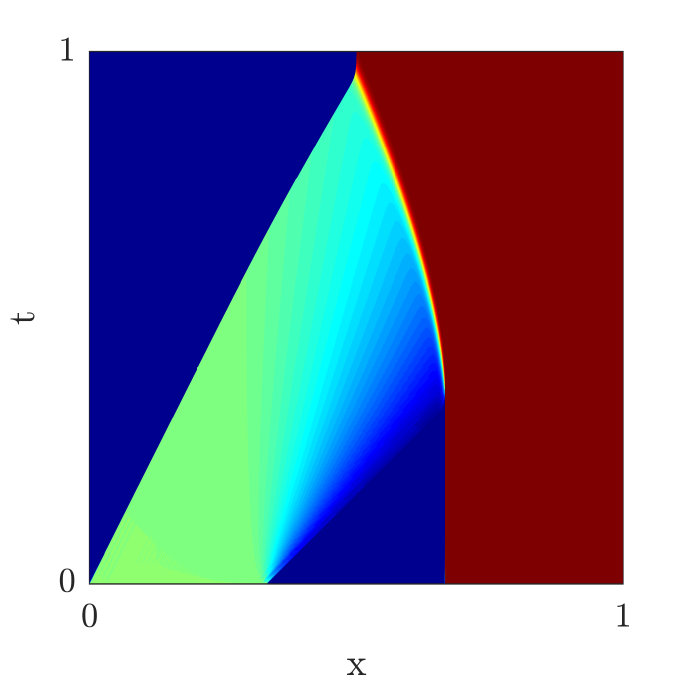}
    \includegraphics[scale=0.8,clip,trim=22 0 15 0]{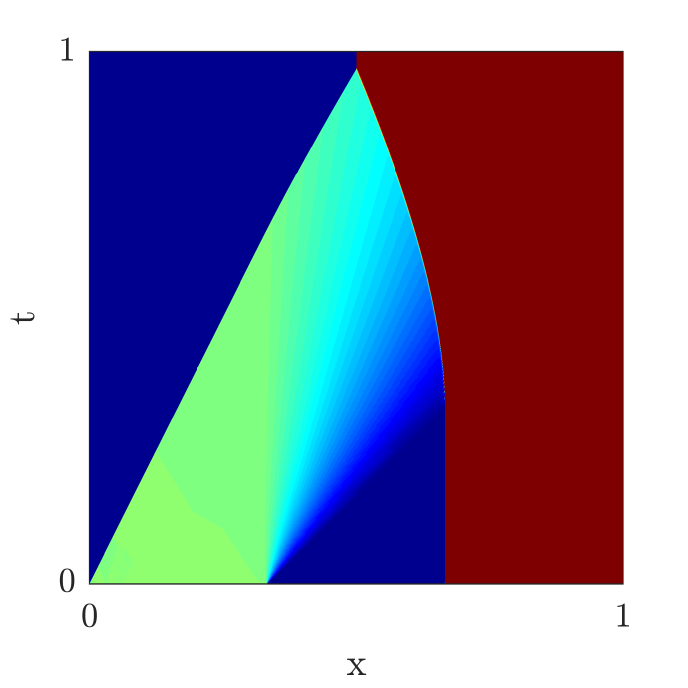}
    
    \caption{Solution of the nonlocal balance law with exponential kernel (\textbf{top}, \cref{defi:W_nonlocal}) and constant kernel (\textbf{bottom}, \cref{eq:const_kernel}) supplemented by the piecewise constant initial datum stated in \cref{eq:ex_initial_datum} plotted in the space-time domain. From left to right $\eta$ is decreasing, \(\eta \in\big\{ 10^{-1},\  10^{-2},10^{-3}\big\}\). The rightmost figure is ``by eye'' not distinguishable from the corresponding local solution. \textbf{Colorbar:} $0\ $\protect\includegraphics[width=2cm]{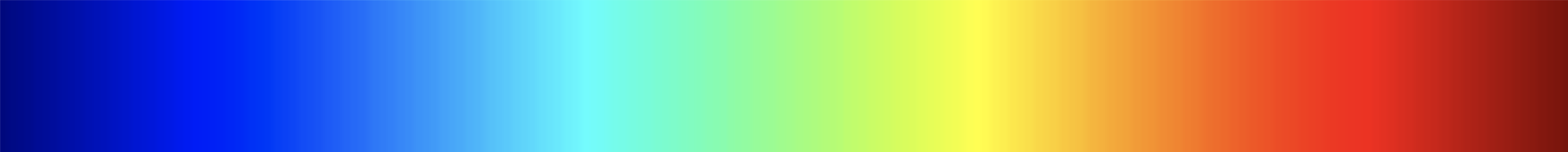}$\ 1$}
    \label{fig:example_3d}
\end{figure}

\begin{figure}
    \centering
    \includegraphics[scale=0.8,clip,trim=0 25 10 10]{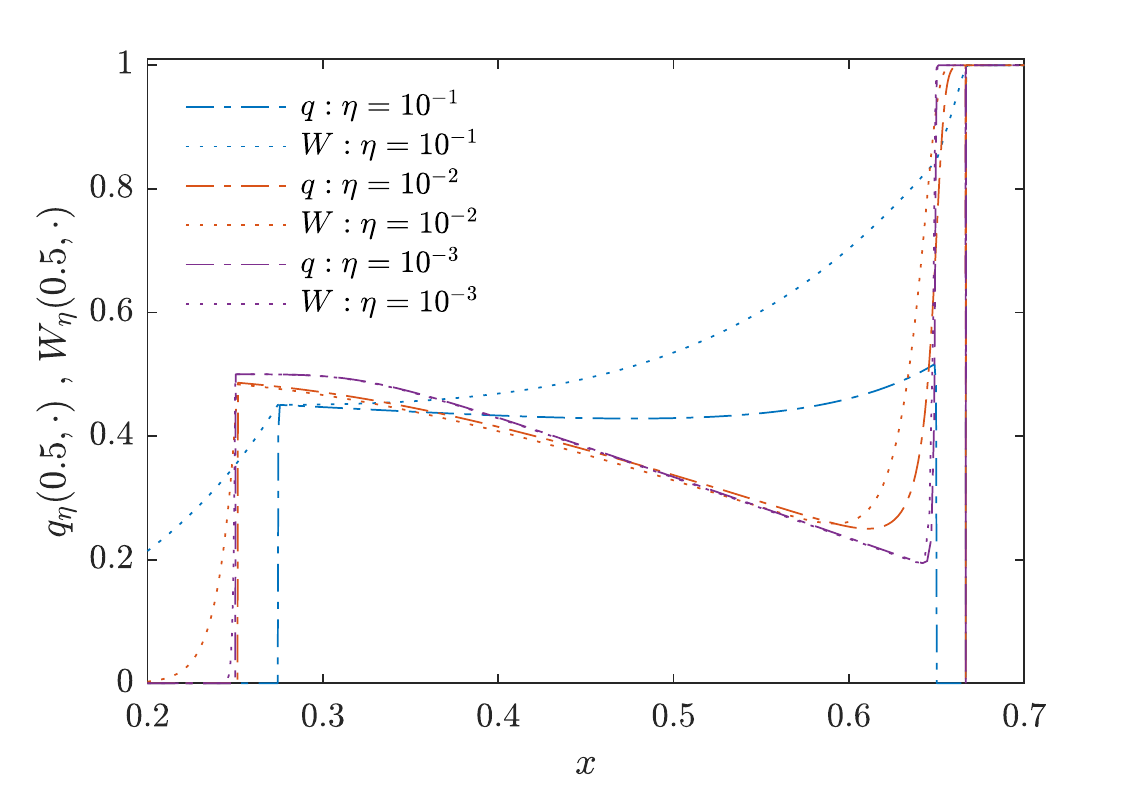}
    \includegraphics[scale=0.8,clip,trim=0 25 10 10]{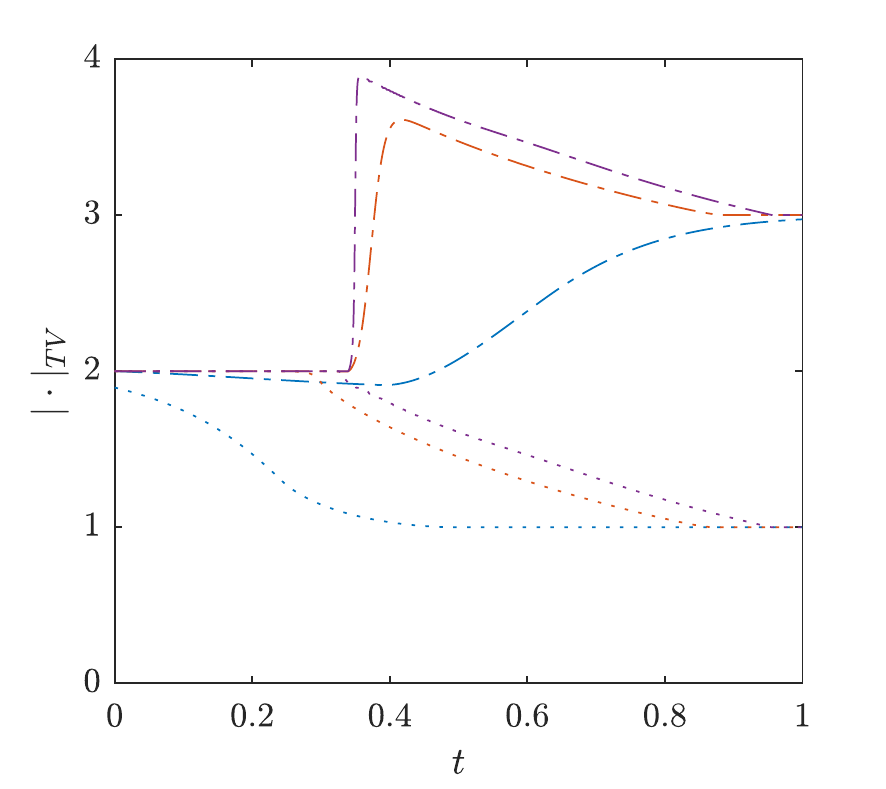}
    
        \includegraphics[scale=0.8,clip,trim=0 0 10 10]{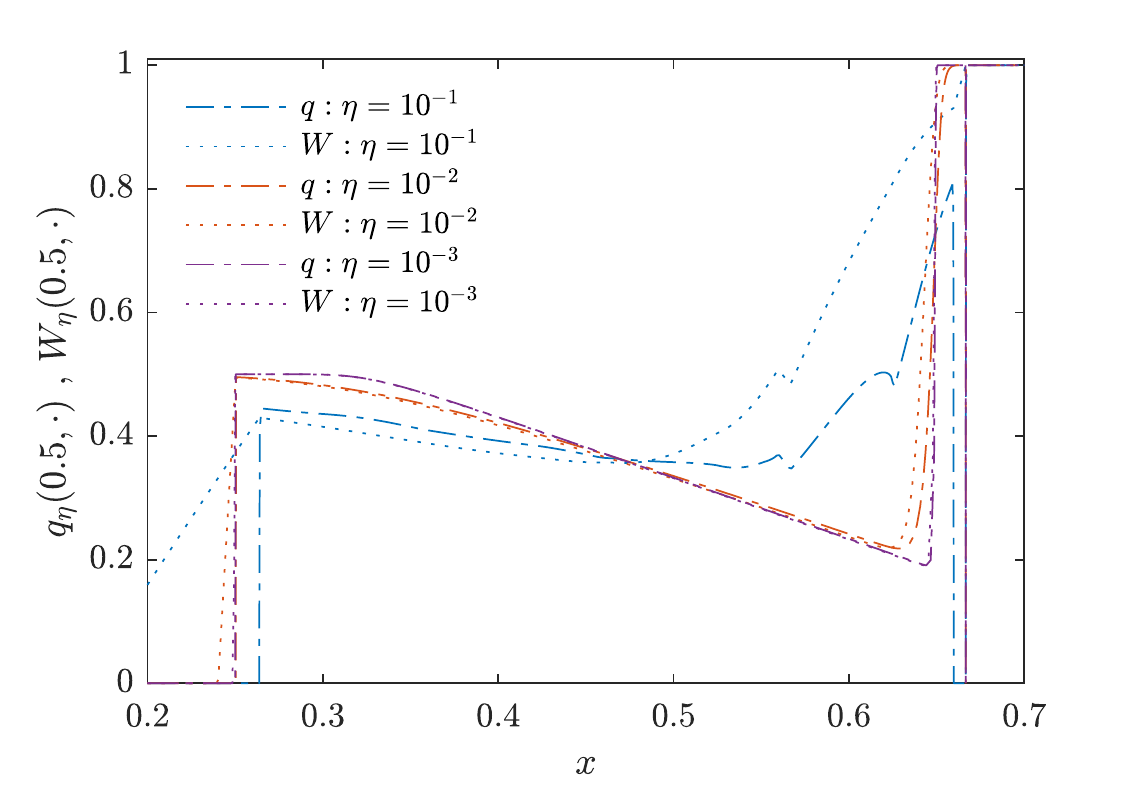}
    \includegraphics[scale=0.8,clip,trim=0 0 10 10]{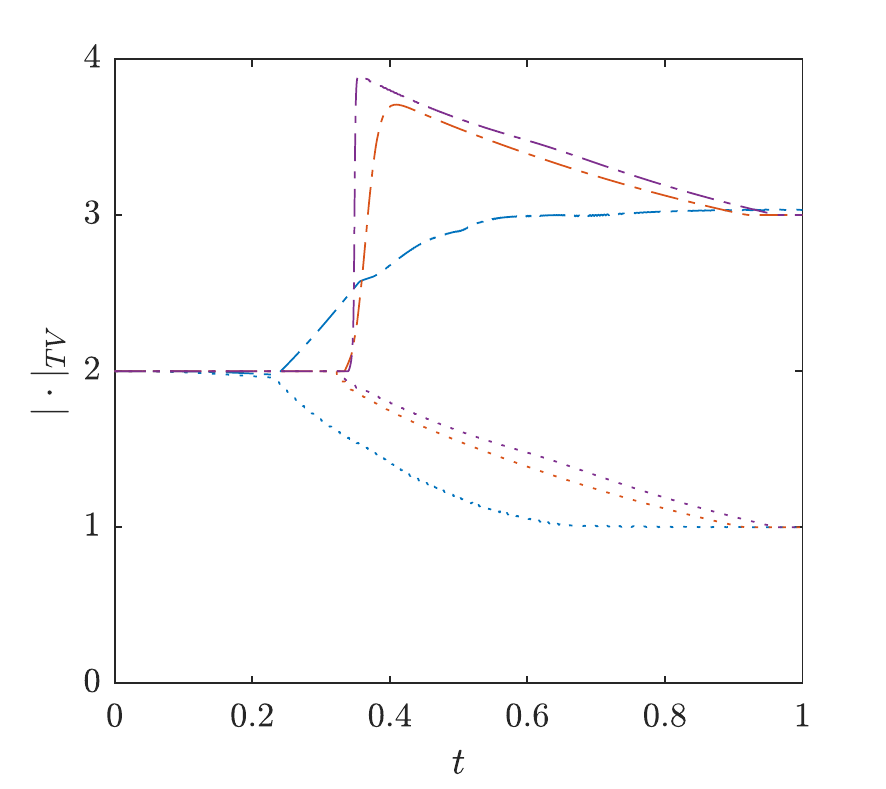}
    
    \caption{\textbf{Left:} Solution of the nonlocal balance law  with exponential kernel (\textbf{top}, \cref{defi:W_nonlocal}) and constant kernel (\textbf{bottom}, \cref{eq:const_kernel}) supplemented by the piecewise constant initial datum stated in \cref{eq:ex_initial_datum} and its corresponding nonlocal term plotted for $t=0.5$ and $\eta \in \{10^{-1},10^{-2},10^{-3}\}$. \textbf{Right:} Evolution of the corresponding total variations showing a monotone decreasing nature in terms of the nonlocal term (dotted lines) which is also the case for the local counterpart. In terms of the total variation of the solution itself (dashed dotted lines), the total variation approaches $3$. This is due to that the zero in the initial datum ($x\in(\tfrac{1}{3},\tfrac{2}{3})$) moves and shrinks but does not vanish for all $\eta \in\R_{>0}$ and $t\in(0,T],$ resulting in an additional total variation of $2$ compared to the total variation of the solution to the local equation being $1$ for all $t\in(1,T)$.}
    \label{fig:examle_2d} 
\end{figure}

Some numerical results concerning the convergence can already be found in \cite{MR3944408}. We rely on a solver based on characteristics \cite{pflug2} which is non dissipative.
On the basis of a simple example we want to shed more light on the difference between the total variation of $q_\eta$ and the nonlocal counterpart $W_\eta[q_\eta]$. We further demonstrate that the result should still hold for general nonlocal kernels by using as ``worst case'' a constant kernel, i.e. for \(q\in C\big([0,T];L^{1}_{\loc}(\R)\big)\cap L^{\infty}((0,T);L^{\infty}(\R))\)
\begin{align}
    W_\eta[q](t,x) \: \tfrac{1}{\eta}\int_{x}^{x+\eta}q(t,y)\dd y,\qquad x\in\R.
    \label{eq:const_kernel}
\end{align}
It seems to be true that a total variation bound on the nonlocal term holds and that also the solution still converges to the local entropy solution.
The following examples rely on the following initial datum:
\begin{align}\label{eq:ex_initial_datum}
    q_0 &\equiv \tfrac{1}{2}\chi_{(0,\frac{1}{3})} + \chi_{\R_{>\frac{2}{3}}}
\end{align}

The crucial point of the chosen initial datum are the roots for $x\in (\tfrac{1}{3},\tfrac{2}{3})$. These roots are moving but kept in the nonlocal solution $q_\eta$ for all times. This results in an increase of the total variation. In the nonlocal term $W$ there are by construction of the initial datum as well as the exponential kernel no roots and the solution is smoothed resulting in a -- as proven -- non-increasing total variation.

\section{Future work}
\label{sec:conclusions}

What remains an open question is whether it is possible to obtain the same results for different kernels still satisfying the required monotonicity assumption for the solution to satisfy a maximum principle (see for this particularly \cref{sec:numerics} and \cref{fig:example_3d} bottom). The considered exponential kernel clearly provides a nice structure which seems to be crucial in our analysis for showing the stated results.

Another interesting problem consists of what happens in the case of a fully symmetric nonlocal kernel which is sensitive to both propagating directions. However, such a kernel immediately implies that the solutions cannot satisfy a maximum principle (for an illustration see for instance \cite[Example 7.3, Fig. 9]{MR3944408}). Then, recalling \cite{1808.03529} it is also apparent that one cannot expect the solution to converge in a strong or weak sense to the entropy solution, but there is hope -- compare particularly the numerics in \cite[Example 7.3]{MR3944408} -- for convergence in a measure valued sense.

\section*{Acknowledgments}
G.~M.~Coclite is a member of the Gruppo Nazionale per l'Analisi Matematica, la Probabilit\`a e le loro Applicazioni (GNAMPA) of the Istituto Nazionale di Alta Matematica (INdAM). He has been partially supported by the  Research Project of National Relevance ``Multiscale Innovative Materials and Structures'' granted by the Italian Ministry of Education, University and Research (MIUR Prin 2017, project code 2017J4EAYB and the Italian Ministry of Education, University and Research under the Programme Department of Excellence Legge 232/2016 (Grant No. CUP - D94I18000260001).

J.-M.~Coron acknowledges funding from the Miller Institute and from the Agence Nationale de La Recherche (ANR), grant ANR Finite4SoS (ANR-15-CE23-0007). He also thanks the Miller Institute and UC Berkeley for their hospitality.

N.~De Nitti has been partially supported by the Alexander von Humboldt Foundation and by the TRR-154 project of the DFG.  

L.~Pflug has been supported by the Deutsche Forschungsgemeinschaft (DFG, German Research Foundation) -- Project-ID 416229255 -- SFB 1411.

\bibliographystyle{abbrv}
\bibliography{NonlocalToLocal-ref}

\end{document}